\documentclass[a4paper]{amsart}

\usepackage{amsmath}
\usepackage{amsthm}
\usepackage{amsfonts}

\usepackage{hyperref}

\usepackage{abstract}

\usepackage[all]{hypcap}[2006/02/20]

\newcommand{\mb}[1]{\mathbf{#1}}

\newcommand{\bb}[1]{\mathbb{#1}}

\newcommand{\pard}[2]{\frac{\partial #1}{\partial #2}}

\newcommand{\ddt}[1]{\frac{d #1}{dt}}

\newcommand{\mbf}[1]{\mathbf{#1}}

\newcommand{\mbb}[1]{\mathbb{#1}}

\newcommand{\ip}[2]{\left \langle #1 , #2 \right\rangle}

\newcommand{\size}[1]{\left | #1  \right|}

\newcommand{\ov}{\overline}

\newcommand{\ti}{\tilde}
\newcommand{\n}{\nabla}

\renewcommand{\l}{\lambda}

\newcommand{\ra}{\rightarrow}

\newcommand{\Si}{\Sigma}

\newcommand{\e}{\epsilon}

\newcommand{\ho}{\left(\frac{d}{dt} - \Delta\right)}

\newcommand{\mh}{\mathfrak{H}}

\newcommand{\mv}{\mathfrak{v}}

\usepackage[pdftex]{graphicx}

\begin{document}

\theoremstyle{plain}

\newtheorem{theorem}{Theorem}

\newtheorem{lemma}[theorem]{Lemma}

\newtheorem{claim}[theorem]{Claim}

\newtheorem{prop}[theorem]{Proposition}

\newtheorem{cor}[theorem]{Corollary}

\theoremstyle{definition}

\newtheorem{assumption}{Condition}

\newtheorem{definition}{Definition}

\newtheorem{example}{Example}

\theoremstyle{remark}
 
\newtheorem{remark}{Remark}

\title{Construction of Maximal Hypersurfaces with Boundary Conditions}

\author{Ben Lambert}

\email{benjamin.lambert@uni-konstanz.de}

\maketitle

\begin{abstract}
We construct maximal hypersurfaces with a Neumann boundary condition in Minkowski space via mean curvature flow. In doing this we give general conditions for long time existence of the flow with boundary conditions with assumptions on the curvature of a  Lorentz boundary manifold. 
\end{abstract}

\section{Introduction and notation}
In this paper we use Mean Curvature Flow (MCF) with a Neumann boundary condition to construct maximal hypersurfaces with boundary in Minkowski space $\bb{R}^{n+1}_1$ for $n\geq2$, which are perpendicular to a given Lorentz surface, $\Sigma$ at their boundary. Maximal surfaces  are well known to be useful in the study of semi-Riemannian manifolds and mathematical relativity. A famous example in which these surfaces play a central part is the first proof of the positive mass conjecture by Schoen--Yau \cite{SchoenYau}. Correspondingly the existence and properties of such surfaces have been well studied, and we do not give a full literature review here. We mention Bartnik \cite{BartnikEntire}, for existence of entire maximal hypersurfaces in asymptotically flat spacetimes, Bartnik and Simon \cite{BartnikSimon} where solvability of the Dirichlet problem in Minkowski space was proven, and Gerhardt \cite{GerhardtHSurfaces} for the existence of foliations of constant mean curvature and the solvability of the Dirichlet problem in curved spacetimes. Ecker and Huisken \cite{EckerHuiskenCMCMink} first used a parabolic prescribed mean curvature flow to construct surfaces of prescribed mean curvature, and the assumptions on ambient manifolds for such flows have been weakened by Gerhardt \cite{GerhardtHSurfacesthesequal}. Conditions for construction of constant mean curvature surfaces in Minkowski space by a perscribed mean curvature flow in the noncompact case has been studied by Aarons\cite{Aarons}. The Dirichlet boundary problem for MCF in spaces of indefinite metric has been considered by Ecker \cite{EckerMinkowskiDBC, EckerNull}.

Neumann boundary conditions for MCF in Euclidean space have been studied in various situations, and many tools of classical MCF singularity analysis now have a Neumann boundary condition couterpart, see for example the works of Stahl \cite{Stahlfirst}\cite{Stahlsecond}, Buckland \cite{Buckland} and Edelen \cite{Edelen}. Graphical Euclidean MCF with a perpendicular Neumann boundary condition has also been studied over compact domains by Huisken \cite{Huiskengraph}, and over halfspaces by Wheeler \cite{WheelerHalfspace} and in the rotationally symmetric case by Wheeler \cite{WheelerRotSym}. Graphs over Killing vector fields have also been considered by Lira and Wanderly \cite{LiraWanderly} and also the author \cite{LambertTorus}. Mixed Neumann and Dirichlet boundary conditions have also been considered by Wheeler and Wheeler \cite{WheelersDirichletNeumann}, see also the rotational case by Wheeler \cite{WheelerRotSym}. MCF with a Neumann boundary condition in Minkowski space has also been investigated by the author in the standard graphical case \cite{LambertConvex} and within a cone boundary manifold \cite{LambertMinkowski}.  

We require two properties of MCF to construct our maximal hypersurfaces, firstly that the flow stays in a bounded region of Minkowski space, and  secondly that the flowing hypersurface remains strictly spacelike (which then implies the flow exists for all time). The first of these may be achieved by assuming the existence of suitable comparison solutions. The second requirement will be proven in the form of a gradient estimate (for similar estimates, see for example \cite{BartnikEntire, EckerHuiskenCMCMink,GerhardtHSurfacesthesequal,EckerHuiskengraphentire,EckerHuiskenInteriorEstimates}) under a curvature assumption on the boundary manifold, which in dimension 2 is akin to mean convexity. We remark that the flow is still interesting in the absence of some of these assumptions, for example, we may get convergence to homothetic solutions (see  \cite{LambertMinkowski}), and that the estimates in this paper may still be of interest in some such situations. If the flow remains in a bounded region, then for any sequence of times  we may find a subsequence $t_i$ such that $M_{t_i}$ converges to a maximal surface. To obtain better convergence, for example convergence of the whole flow, we need to assume that the maximal surface is stable under the flow, see the final section of this paper for a discussion of stability issues.

Suppose $\Sigma\subset \bb{R}^{n+1}_1$ is a semi-Riemannian hypersurface with a spacelike unit normal $\mu$ and $M^n$ a compact manifold with boundary $\partial M$. We suppose we are given $\mb{F}_0:M^n \ra \bb{R}^{n+1}_1$, an initial spacelike embedding of $M^n$ such that $\mb{F}_0(\partial M^n)\subset \Sigma$. Let $\mathbf{F}: M^n \times [0,T) \rightarrow \bb{R}^{n+1}_1$ be such that
\begin{equation}
\label{MCF}
\begin{cases}
\frac{d \mathbf{F}}{dt} = \mathbf{H}= H \nu & \forall (x,t) \in M^n \times [0,T]\\
\mathbf{F}(\cdot,0)= \mb{F}_0&\\
\mathbf{F}(x,t) \subset \Sigma & \forall (x,t) \in \partial M^n \times [0,T]\\
\ip{\nu}{\mu \circ \mb{F}}(x,t)=0 & \forall (x,t) \in \partial M^n \times [0,T]\ \ ,\\
\end{cases}
\end{equation}
then $\mathbf{F}$ moves by \emph{Mean Curvature Flow with a Neumann free boundary condition $\Sigma$} (here $\nu(x,t)$ is the normal to $\mbf{F}$ at time $t$, and $H$ is the mean curvature with respect to $\nu$).  As is standard, we will write $M_t$ for the image of $F(\cdot, t)$. From here onwards we will assume that $\Sigma$ is topologically a cylinder, and $M^n$ is topologically a $n$-ball. We will also assume that $\mb{F}_0$ satisfies the compatibility condition that at the boundary $\ip{\nu|_{t=0}}{\mu \circ \mb{F}_0} =0$.  

We will need various geometric quantities on various manifolds. A bar will imply quantities on $\mbb{R}^{n+1}_1$, for example $\overline \Delta, \overline \nabla, \ldots$ and so on; no extra markings $\Delta, \nabla, \ldots $ will refer to geometric quantities on $M_t$ our flowing surface at time $t$ and for any other manifold $Z$ $ \Delta^Z,  \nabla^Z, \ldots \text{etc.}$ will refer to the Laplacian, covariant derivatives, $\ldots$  on $Z$.

We state the main theorem of this paper:
\begin{theorem}
 Suppose that $\Sigma$ satisfies Conditions \ref{cassump} and \ref{uassump} below, and $\mb{F}_0$ is a smooth, spacelike, compatible initial embedding. Suppose there exist comparison solutions such that the flowing hypersurface $M_t$ remains in a compact region of $\bb{R}^{n+1}_1$. Then a solution to (\ref{MCF}) exists for $T=\infty$ which is smooth with uniform bounds on all derivatives. Furthermore there exists a sequence $t_i \ra \infty$ such that $M_{t_i}\ra M_\infty$ where $M_\infty$ is a smooth maximal surface satisfying the boundary condition. If for all $p\in\partial M_\infty$, $A^\Si(\nu^\infty, \nu^\infty)|_p>0$ then the whole flow converges to $M_\infty$ in the sense that $M_t \ra M_\infty$ smoothly as $t\ra \infty$.
\end{theorem}

The Theorem is proven as follows: In Section \ref{RP} we show that the above flow is equivalent to a quasilinear PDE, which leads to short time existence for the flow, Proposition \ref{shorttime}, and indicates that the key to obtaining the long time existence above is a suitable gradient estimate. In Section \ref{CS} we determine what constitutes a comparison solution with boundary conditions, see equation (\ref{CompSol}) and Proposition \ref{Comparison}. In Section \ref{EE} we collect the necessary evolution equations and boundary derivatives. In Section \ref{GE} we use an iteration argument to prove suitable estimates on the mean curvature culminating in Proposition \ref{Hest}. We then use Proposition \ref{Hest} to prove the gradient estimate, Theorem \ref{Gradest}, which demonstrates that the above flow exists for all time and is uniformly smooth, see Corollary \ref{boundedflow}. In Section \ref{CT} we prove sequential convergence and construct comparison solutions to give conditions for stability of maximal surfaces under MCF, which then give convergence criteria for the whole flow, see Lemma \ref{FullStab} and Corollary \ref{Stabcondition2}.

In the case that the flow does not stay in a compact region, we may still use the estimates obtained to infer long time existence, see for example Corollary \ref{unboundedflow} which states that if $\Si$ satisfies Conditions \ref{cassump} and \ref{uassump} below, either a solution to (\ref{MCF}) exists for $T=\infty$, or its graph function becomes unbounded in finite time. 
 
We must assume some bounds on the geometry of $\Sigma$. In the absence of any such assumptions we may construct the following example of singular behaviour: In $\bb{R}^2_1$ we parametrise a ``death's trumpet'' boundary manifold $\Sigma$ given graphically by $y=\log \sinh |x|$. $\Si$ has been chosen so that the Minkowski equivalent of the grim reaper solution to MCF given by $u(x,t)= \log \cosh x +t$ is perpendicular to $\Si$ at all $(x,t)$ such that $u$ and $\Si$ intersect. Then starting at any negative time the grim reaper gives the solution to (\ref{MCF}) in Figure \ref{deathstrumpet}.
\begin{figure}[h!]
\begin{center}
 \label{deathstrumpet}
\end{center}
\includegraphics[scale=0.3]{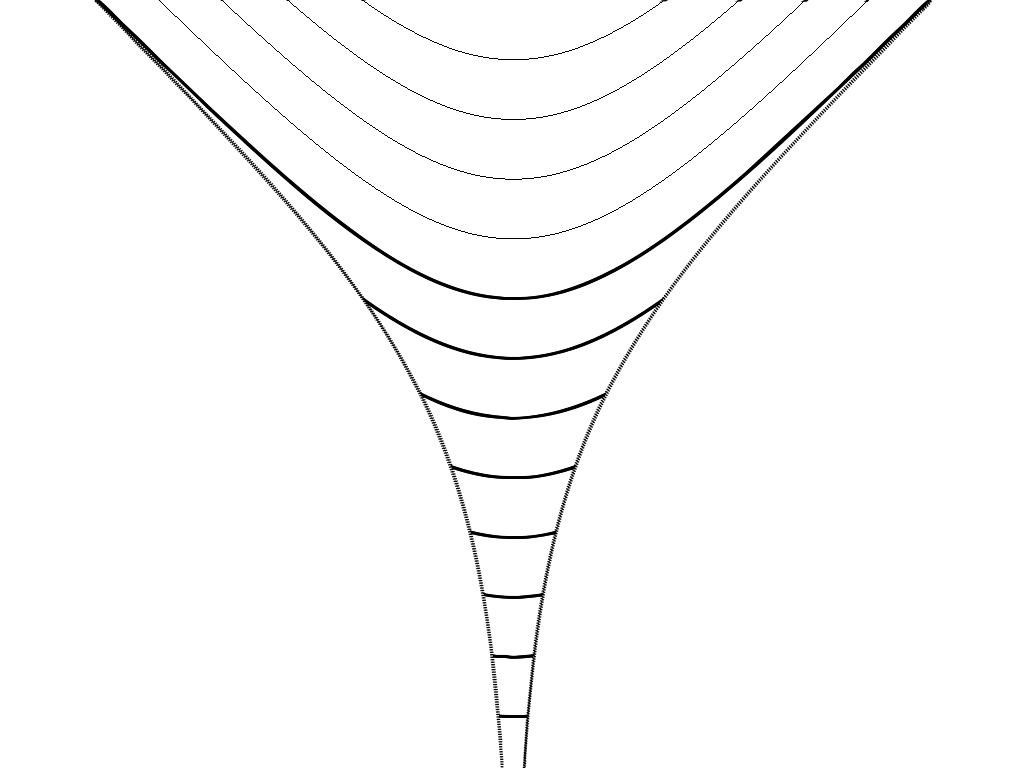}
\caption{Grim reaper solution moving inside the death's trumpet boundary.}
\end{figure}
At time $t=0$ we see that this solution is tangent to the light cone at infinity, and the Neumann boundary condition is no longer defined. We are able to continue the flow for $t>0$ on the interior but we no longer have a boundary to speak of and the flowing manifold is no longer strictly spacelike.

We now define our curvature conditions on $\Si$. We agree that the signs on the second fundamental form on $M_t$ and $\Si$ are given by $A(V,W) = \ip{\ov \n_V \nu }{W}$ and $A^\Si(V,W) = \ip{\ov \n_V \mu }{W}$ respectively. We will also sometimes write $A(\pard{\mb{F}}{x^i}, \pard{\mb{F}}{x^j}) = h_{ij}$, and $|A|$ for the tensor norm of $A(\cdot, \cdot)$. One possible condition we could impose on $\Sigma$ is convexity, and this immediately allows application of a maximum principle to get a spacelike flow, but is extremely restrictive in terms of allowed $\Sigma$. Instead we assume the following weaker curvature conditions:

\begin{assumption}[Curvature assumptions on $\Sigma$]\label{cassump}
The curvature of $\Sigma$ is uniformly bounded and there exists a smooth timelike unit vector field $V$ on $\bb{R}^{n+1}_1$, such that for all $p\in \Sigma$
\begin{enumerate}
 \item $V(p)\in T_p\Si$, 
 \item $V(p)$ is an eigenvector of the second fundamental form of $A^\Sigma(\cdot, \cdot)|_p$, and
 \item $\ov\n_\mu V|_p=0$ .
\end{enumerate}
At a point $p\in\Sigma$, let $W_I$ for $1\leq I \leq n-1$ be the remaining (spacelike) eigenvectors of $A^\Sigma(\cdot,\cdot)$. We assume that for $1\leq I\leq n-1$ the curvature satisfies
\[A^\Si(W_I, W_I)+A^\Si(V, V) \geq 0\ \ .\]
\end{assumption}

This allows significantly more varied boundary manifolds than a convexity assumption, and is similar to 2-convexity. 

We define $\hat{\Si}$ to be the open region of $\bb{R}^{n+1}_1$ such that $\partial \hat \Si = \Si$ and $\mu$ points out of $\hat \Si$. We will require coordinates on $\hat{\Si}$:

\begin{definition} 
We define a smooth diffeomorphism $\ov F :\Omega\times\bb{R}\rightarrow \hat{\Sigma} \subset \bb{R}_1^{n+1}$, where $\Omega\subset\subset \bb{R}^n$ is open and bounded with smooth boundary $\partial \Omega$, to be a \emph{spacelike foliation compatible with the boundary} if:\label{intro}
\begin{enumerate}
 \item The image of $\partial \Omega\times \bb{R}$ under $\ov{F}$ is $\Sigma$.
 \item Let $x^i, \ i=1, \ldots, n$ be coordinates on $\Omega$ and let $\lambda$ parametrize $\bb{R}$ then we assume $\ip{\pard{\ov{F}}{\l}}{\pard{\ov F}{x^i}}=0$ and that $\ov F(\cdot, \lambda)$ is a spacelike hypersurface with normal in the timelike direction $\pard{\ov F}{\l}$, and there exists a uniform constant $C_{\ov F}>0$ such that $-\size{\pard{\ov F}{\l}}^2>C_{\ov F}>0$. 
\item  If $\gamma$ is the outward unit normal to $\partial \Omega$ then $\gamma^i \pard{\ov F}{x^i}$ is in the direction $\mu$.
\item All geometric quantities on the hypersurfaces $\ov F (\cdot, \l)$, for example positivity of the metric and bounds on the curvature, may be uniformly bounded in $\l$.
\end{enumerate}
\end{definition}
In section \ref{CT} we explicitly calculate examples of compatible foliations in the case that $\Si$ is rotationally symmetric.

Given a compatible foliation as above, one may construct the smooth time function $\tau:\bb{R}^{n+1}_1\ra \bb{R}$ defined by $\tau(y) = P(\ov{F}^{-1}(y))$ where $P: \Omega \times \bb{R} \rightarrow \bb{R}$ is the standard projection. Such a $\tau$ satisfies $\ov \n_\mu \tau = 0$ on $\Sigma$, and in fact $\ov{\n} \tau = \size{\pard{\ov F}{\l}}^{-2}\pard{\ov F}{\l}$. We will write the lapse function $\psi= \sqrt{-\size{\pard{\ov{F}}{\l}}^2}$. 

For any compatible spacelike foliation, we define the normal vector field \[\hat{V}:=\psi^{-1}\pard{\ov F}{\l}\ \ .\] 

\begin{assumption}[Existence of a compatible foliation]\label{uassump}
There exists a spacelike foliation compatible with the boundary such that there exists a constant $C_V>1$ such that $1\leq|\ip{V}{\hat{V}}|\leq C_V$, where $V$ is the unit vector field from Condition \ref{cassump}.
\end{assumption}

We define two notions of gradient, $v=-\ip{V}{\nu}$ and $\hat{v}=-\ip{\hat{V}}{\nu}$, where we choose signs on $V$ and $\hat{V}$ such that these functions are both positive. 
\begin{remark}\label{remarkablev}
Due to the above condition, it is easy to see that there exists a $\tilde{C}_V$ depending only on $C_V$ such that
\[\frac{1}{\tilde{C}_V }v\leq \hat{v} \leq \tilde{C}_V v\ \ .\]
\end{remark}

\begin{remark}\label{Tensorineq}
 We observe that as in \cite[Equation (3)]{EckerHuiskenCMCMink} if $T$ is any $p$-tensor defined on $\bb{R}^{n+1}_1$ and $T_M$ is the restriction of $T$ to $M_t$, we may estimate $\left|T_M\right|\leq v^p|T|_{\bb{R}_1^{n+1}}$. 
\end{remark}

To obtain a good gradient estimate in settings where the flow does not stay in a bounded region, we will also consider:
\begin{assumption}[Boundedness of maximum volume]\label{finitearea}
The maximum volume of a spacelike hypersurface with boundary on $\Sigma$ is bounded above by $C_{\text{vol}}< \infty$. 
\end{assumption}

Due to the spacelikeness of $\mu$, Condition \ref{finitearea} automatically holds while the flow stays in a bounded region. However this means that for $\Sigma$ which are tangent to cones at infinity our gradient estimate in Theorem \ref{Gradest} gets worse as the solution moves towards spatial infinity.

\begin{remark}
We note that the counter example in Figure \ref{deathstrumpet} violates both Conditions \ref{cassump} and \ref{finitearea}.
\end{remark}

The author would like to thank the reviewer for their useful comments and suggestions.

\section{Rewriting the problem}\label{coordinates}
\label{RP}
We consider coordinates given by $\ov F$, a compatible spacelike foliation as in the previous section. Writing $i$ for the $x^i$th coordinate on $\Omega$ and $\hat{g}_{ij}(\lambda)$ is the metric of the hypersurface defined by $\ov F(\cdot, \lambda)$ we obtain $ \ov{g}_{ij}=\hat{g}_{ij}(\lambda), \ \ov{g}_{i\l}=0$ and \linebreak$\ov{g}_{\l\l}= -\psi^{2}<0$. We now write a general spacelike hypersurface $\tilde{M}\subset\hat{\Sigma}$ graphically where we parametrise $\tilde M$ by $G(x)=\ov{F}(x, \phi(x))$. We then have that the metric and its inverse are given by
\[g_{ij} = \hat{g}_{ij} -\psi^2D_i \phi D_j \phi, \text{ and } g^{ij}=\hat{g}^{ij} + \hat{v}^{2}\psi^2 D_p\phi \hat{g}^{pi} D_q \phi \hat{g}^{qj}\ \ ,\]
where $\hat{v}^{-1}=\sqrt{1-\psi^2 D_r\phi \hat{g}^{rs}D_s\phi}$. The gradient quantity $\hat{v}$ is the same as in the previous section, i.e. $\hat{v}=-\ip{\hat{V}}{\nu}$.  We calculate the volume form to be
\begin{equation} \sqrt{\det{g_{ij}}}=\sqrt{\hat{v}^{-2}\det{\hat{g}_{ij}(x, \phi(x))}}\ \ ,
 \label{volformeq}
\end{equation}
and note that the ``future directed'' (that is in the same direction as $\pard{\ov F}{\l}$) unit normal may be written as
\[\nu=\hat{v}\left[\psi D_k\phi \hat{g}^{kp}\pard{\ov F}{x^p}+ \psi^{-1}\pard{\ov F}{\l}\right]\ \ .\]

Any function $f$ on $M$ may also be written as a function on $\Omega$. As such we may calculate that 
\begin{equation}\size{\n f}^2 = D_i f \hat{g}^{ij} D_j f +\psi^2 \hat{v}^2(D_if \hat{g}^{ij} D_j \phi)^2\geq  D_i f \hat{g}^{ij} D_j f \geq C|Df|^2
 \label{Dfeq}
\end{equation}
where $C$ depends only on $\ov F$. We use this to obtain integral estimates, which are necessary since to the author's knowledge there is no equivalent of the Michael--Simon Sobolev inequality in Minkowski space. We obtain boundary and Sobolev inequalities on our flowing manifold by simply using the Euclidean equivalents on $\Omega$. Of course these estimates are not coordinate invariant and so include factors of $v$, but they are sufficient for our purposes. 

\begin{lemma}\label{intlemmas}
Suppose $\Sigma$ satisfies Condition \ref{uassump}. Let $f\in C^1$ be a positive function on a spacelike hypersurface $\ti M$ inside $\hat{\Sigma}$ with $\partial \ti M \subset \Sigma$ such that at the boundary $\ip{\nu_{\tilde{M}}}{\mu}=0$. Then we may estimate 
\[ \int_{\partial \ti{M}} f dV^{\partial}  \leq  C_2\int_{\ti M} |\n f| +f(|A|+ \hat{v})dV \] 
and 
\[\left(\int_{\ti M} |f|^\frac{n}{n-1}dV\right)^\frac{n-1}{n} \leq C_1 \underset{x\in \ti M} \sup \hat{v} \int_{\ti M} |\n f| +|f| dV\]
for constants $C_1, C_2$ depending only on $n$, $\Sigma$ and $\ov{F}$.
\end{lemma}
\begin{proof}
We consider the hypersurface $\ti M$ written graphically as $\ov{F}(x, \phi(x))$ and write $C_n$ for any constant that depends only on $\Sigma$, $\ov F$, $n$. From properties of a compatible foliation, we have $\mu = S(x,\l) \gamma^i \pard{\ov F}{x^i}$, where $S(x,\l)>0$, so that the boundary condition on $\partial M$ becomes 
\begin{equation}0=\ip{\nu}{\mu} = vS\psi\gamma^k D_k\phi\ \ ,
 \label{boundaryphi1}
\end{equation}
 that is, $\gamma^kD_k\phi=0$. Under such a condition we may see that the boundary volume form on $\partial M$ may be written in $\ov F$-coordinates as $\hat{v} \sqrt{\det\hat{g}^\partial_{ij}(x, \phi)}$. 

We use Remark \ref{Tensorineq} to estimate 
\begin{equation}|\n \hat{v}|^2\leq C_n(|A|^2\hat{v}^2+\hat{v}^4)\ \ .\label{gradv}
\end{equation}
We may now apply \cite[Lemma 1.4]{Gerhardt} on $\Omega$ to see that, for $0\leq f\in C^1(M)$
\begin{flalign*}
 \int_{\partial M} f dV^{\partial} & \leq C_n \int_{\partial \Omega} \frac{f}{\hat{v}} dS\\
&\leq C_n \int_\Omega \left[|Df| +f\frac{|D\hat{v}|}{\hat{v}} + f\right] \frac{1}{\hat{v}}dx\\
&\leq C_n\int_M |\n f| +f(|A|+ \hat{v})dV\ \ ,
\end{flalign*}
where we estimated using (\ref{Dfeq}) and (\ref{gradv}).

For the second inequality we may use the uniform boundedness of $\det{\hat{g}_{ij}}$, equations (\ref{volformeq}) and (\ref{Dfeq}), to see that for $f\in C^1(\Omega)$ 
\begin{flalign*}
 \left(\int_M |f|^\frac{n}{n-1}dV\right)^\frac{n-1}{n}&\leq C_n \left(\int_\Omega |f|^\frac{n}{n-1}dx\right)^\frac{n-1}{n}\\
&\leq C_n \int_\Omega|Df| +|f|dx\\
&\leq C_n \underset{x\in M} \sup \hat{v} \int_M |\n f| +|f| dV\ \ .
\end{flalign*}
where the second inequality follows from \cite[Lemma 1.1 and Lemma 1.4]{Gerhardt}.
\end{proof}

\begin{remark}
Using Remark \ref{remarkablev} we see that Lemma \ref{intlemmas} still holds if we exchange $\hat{v}$ for $v$ (although with different constants).
\end{remark}

We now add a time dependence so that $G(x, t)=\ov{F}(x, \phi(x, t))$, and rewrite mean curvature flow in terms of $\phi(x,t)$. Standard calculations and equation (\ref{boundaryphi1}) then imply that (as in \cite[Section 2]{Stahlfirst}) equation (\ref{MCF}) is equivalent to finding $\phi:\Omega\times[0,T) \ra \bb{R}$ such that
%Easy calculations give
%\[h_{ij} =  \hat{v}\psi D^2_{ij}\phi + \tilde{b}_{ij}(x, \phi, D\phi)\ \ ,\]
%while the (reparametrized) mean curvature flow we have
%\begin{equation}-H = \ip{\ddt{}\ov F(x, \phi(x,t))}{\nu} = \ip{\pard{\ov F}{\l}D_t \phi}{\nu}=-\hat{v}\psi D_t \phi\ \ .
% \label{arbitraryname}
%\end{equation}
%Therefore, using equations (\ref{boundaryphi1}) and (\ref{arbitraryname}), we see that (as in \cite[Section 2]{Stahlfirst}) equation (\ref{MCF}) is equivalent to finding $\phi:\Omega\times[0,T) \ra \bb{R}$ such that
\begin{equation}
 \label{graphMCF}
\begin{cases}
 D_t \phi = g^{ij}(x,\phi,D\phi)D^2_{ij} \phi + b(x,\phi,D\phi)& \text{for } (x,t)\in \Omega \times [0,T)\\
 \gamma^i D_i \phi = 0 & \text{for } (x,t)\in \partial\Omega \times [0,T)\\
\phi(\cdot, 0) = \phi_0(\cdot)&\qquad\qquad\qquad\qquad\qquad\ \ .
\end{cases}
\end{equation}
where $\phi_0:\Omega \ra \bb{R}$ is chosen such that $\ov F(x, \phi_0(x))$ parametrises $M_0$. We remark that equation (\ref{graphMCF}) is a quasilinear parabolic equation, and the main challenge to show long time existence will be to show that it is uniformly parabolic. From the explicit form of $g^{ij}$ above, as is standard in graphical MCF \cite{EckerHuiskenCMCMink,EckerMinkowskiDBC,Huiskengraph,LambertMinkowski,BartnikEntire,BartnikSimon}, this is equivalent to finding an upper bound on the quantity $\hat{v}$, or from Remark \ref{remarkablev} on the quantity $v$. We obtain the following:
\begin{prop}\label{shorttime}
 Suppose $\Si$ has a compatible spacelike foliation and $\mb{F}_0$ is smooth, compatible initial data. Then there exists an $\epsilon>0$ such that a smooth solution to (\ref{MCF}) exists for $T=\epsilon$.
\end{prop}
\begin{proof}
 From the above argument, the statement is equivalent to the existence of a solution $\phi:\Omega\times[0, \epsilon) \ra \bb{R}$ to equation \ref{graphMCF}. Since this is a quasilinear equation with a \emph{linear} boundary condition, this is covered by the standard theory, for example by a trivial modification of \cite[Theorem 8.2, p206]{Lieberman}.
\end{proof}

\section{Comparison solutions}\label{CS}
Throughout this section, let $\Omega\subset \bb{R}^n$ be a domain with smooth boundary $\partial \Omega$. Let $\mb{G}$ be a smooth mapping $\mb{G}:\Omega \times [0,T)\rightarrow \bb{R}^{n+1}_1$ such that $\mb G(\partial \Omega, \cdot)\subset \Si$. Define $N_t$ to be the image of $\mb{G}(\cdot,t)$ where we will assume throughout that $N_t$ is spacelike. Then $\mb{G}$ is a comparison solution from below (above) if for any solution  $M_t$ of (\ref{MCF}) such that $M_0$ is above (below) $N_0$, $M_t$ is above (below) $N_t$ for all $t\in[0,T)$.

Suppose $M_t$ is above $N_t$ and let $\nu_G$ be the upward unit normal of $N_t$%which points in the direction of $M_t$
, and let $H^G$ be the mean curvature calculated with respect to $\nu^G$. We aim to show that if $\mb{G}$ satisfies
\begin{equation}
 \label{CompSol}
\begin{cases}
\ip{\frac{d \mathbf{G}}{dt}}{\nu^G} \geq -H^G  & \forall (x,t) \in \Omega \times [0,T)\\
\mathbf{G}(\cdot,0)= G_0&\\
\mathbf{G}(x,t) \subset \Sigma & \forall (x,t) \in \partial \Omega \times [0,T)\\
\ip{\nu^G}{\mu\circ \mb{G}}(x,t)\leq0 & \forall (x,t) \in \partial \Omega \times [0,T)\\
\end{cases}
\end{equation}
then $\mb{G}$ is a comparison solution from below. 
The proof of this is very similar to Stahl's proof in the Euclidean setting \cite{Stahlfirst}, with some simplifications due to the geometry of Minkowski space. %We require the following maximum principle:
%\begin{prop}[Strong Maximum Principle]\label{SMP}
%Let $\epsilon>0$ be a small constant, $\Omega\subset \bb{R}^n$ a compact, connected domain with smooth boundary $\partial \Omega$ and outward pointing normal $\gamma$. Suppose $\phi:\Omega\times [0,T) \rightarrow \bb{R}$ be a continuous function of class $C^{2;1}(\Omega\times[0,T))$ in the neighbourhood of all point $(x,t)\in \Omega \times[0,T)$ with $|\phi(x,t)|<\epsilon$ such that 
%\begin{equation*}
% \begin{cases}
%  L(\phi)\geq 0 &\forall (x,t)\in \Omega \times[0,T)\\
% \gamma^iD_i \phi \geq 0 & \forall (x,t)\in \partial \Omega \times[0,T)\\
% \phi(\cdot, 0)\geq 0& \qquad\qquad\qquad\qquad\quad,
% \end{cases}
%\end{equation*}
%where $a^{ij}(x,t), b^i(x,t), c(x,t)\in L^\infty(\Omega \times[0,T))$ and \[L(\phi)=\pard{\phi}{t} - a^{ij}(x,t)D^2_{ij}\phi -b^i(x,t)D_i \phi - c(x,t)\phi\ \ .\] Then $\phi\geq 0$ for all $(x,t)\in \Omega \times [0,T)$. Furthermore if initially $\phi(\cdot,0)$ is not identically zero then $\phi(x,t)>0$ for all $(x,t)\in \Omega\times(0,T)$.
%\end{prop}
%\begin{proof}
%For a slightly more general maximum principle see \cite[Theorem 3.1, Corollary 3.2]{Stahlfirst}. 
%\end{proof}

%We have the following:
\begin{prop}\label{Comparison}
 Suppose $\Si$ satisfies Condition \ref{uassump} and we have smooth, spacelike solutions $\mb{F}$ of equation (\ref{MCF}) and $\mb{G}$ of equation (\ref{CompSol}) on a time interval $[0,T)$ such that $M_0$ is contained in the closure of one of the connected components of $\hat \Si \setminus N_0$. Then either $M_t \equiv N_t$ for all $t\in[0,T)$ or $M_t \cap N_t =\emptyset$ for all $t>0$. 
\end{prop}
\begin{proof}
 We consider $\mb{F}$ and $\mb{G}$ in coordinates $\ov F$ inside $\Sigma$ as in the previous section, and we write them as (smooth) graphs $u(x,t)$ and $w(x,t)$ respectively. Since initially $\mb{F}_0$ lies on one side of $\mb{G}_0$, without loss of generality we may assume that $u\geq w$ initially and that $\nu^G$ is an upwards pointing unit vector field. As in the calculations in the previous section we see that 
\begin{flalign*}
 u_t &=g^{ij}(x,u,Du)D^2_{ij} u + b(x,u,Du)\\
 w_t &\leq g^{ij}(x,w,Dw)D^2_{ij} w + b(x,w,Dw)
\end{flalign*}
while at the boundary,
\[\gamma^iD_iu = 0 , \ \ \gamma^i D_iw\leq 0 \]
Writing $\phi = u-w$ then by standard methods we may write
\[\phi_t  \geq a^{ij}(x,t)D_{ij}^2 \phi +b^i(x,t)D_i \phi +c(x,t) \phi, \ \ \gamma^iD_i\phi \geq 0\]
where $a^{ij}(x,t), b^i(x,t), c(x,t) \in L^\infty(\Omega \times [0,T))$. Since $\phi(\cdot, 0)\geq0$, when $\phi$ is small (i.e. when $\mb{F}$ and $\mb{G}$ are close together or touching) we may apply a strong maximum principle of Stahl \cite[Theorem 3.1, Corollary 3.2]{Stahlfirst}, to complete the proof. 
\end{proof}

\section{Evolution equations and boundary identities}\label{EE}
In this section we collect the necessary evolution equations and boundary identities. Firstly, we need standard evolution equations for evolution of the metric and normal:
\begin{lemma}
 On the interior of $M$ we have that \label{evolnu}
\begin{align*}
\ddt{\nu} &= \n H\ \ , \\ 
\ddt{g_{ij}}&=2Hh_{ij}\ \ ,\\ 
\ho H &=- H |A|^2\ \ .
\end{align*}
\end{lemma}
\begin{proof}
 See \cite[Proposition 3.1, Proposition 3.3]{EckerHuiskenCMCMink}.
\end{proof}

From the spatial and time derivatives of the boundary condition we have:
\begin{lemma}
\label{bdryderivs} 
For $p \in \partial M^n \times [0,T)$ and $W \in T_p M_t \cap T_p \Sigma$ then
\[ A(\mu, W)= -A^\Sigma(\nu, W)\ \ .\]
and also
\[\nabla_\mu H=-HA^\Sigma(\nu, \nu)\ \ .\]
\end{lemma}
\begin{proof}
 The is identical to the Euclidean case of Stahl \cite[Proposition 2.1, Proposition 2.2]{Stahlsecond}, see also \cite[Lemma 5.2, Lemma 5.4]{LambertMinkowski}.
\end{proof}

Importantly we will also need the evolution equation for $v=-\ip{V}{\nu}$. 
\begin{lemma}
 On the interior of the flowing manifold, \label{evolv}
\[\ho v = -v|A|^2 + 2g^{ij}A(\ov{\n}_i V, j)+g^{ij}\ip{\ov \n^2_{ij} V}{\nu}\]
holds.
\end{lemma}
\begin{proof}
We calculate from Lemma \ref{evolnu}
\[\ddt{v} =  - \n_{V^\top} H - H \ip{\ov \n_\nu V}{\nu}\]
and
\begin{flalign*}
 \Delta v &= -g^{ij}\left(\ip{\ov \n^2_{ij} V}{\nu} + 2A(i, (\ov \n_j V)^\top)  + \n_{V^{\top}} h_{ij} \right.\\
&\qquad\qquad\qquad\qquad\qquad\qquad\qquad \left. +h_{ir}g^{rk}h_{kj}\ip{\nu}{V} -  \ip{\ov \n_{\n_i j -\ov \n_i j  } V}{\nu}\right)\\ 
& = -g^{ij} \ip{\ov \n^2_{ij} V}{\nu} -2g^{ij}A(i, (\ov \n_j V)^\top) +v|A|^2 -\n_{V^{\top}} H - H \ip{\ov \n_\nu V}{\nu}\ ,
\end{flalign*}
where we used the Codazzi--Mainardi and Weingarten formulae. 
\end{proof}

\begin{lemma} \label{evolu}
We define the function $u:M\ra \bb{R}$ by $u = \tau(F(x,t))$, then 
\[\ho u  = - g^{ij}\ov \n^2_{ij} \tau \]
and we furthermore remark that 
\[|\n u|^2 = \psi^{-2}(\hat{v}^2 -1)\]
\end{lemma}
\begin{proof}
We calculate for a general ambient function $u$
\[\ddt{u} = H \ov \n_\nu u, \ \ \Delta u=g^{ij}\ov \n^2_{ij} u +H\ov \n_\nu u\ \ .\]
Now since $\ov \n u $ is strictly timelike, we calculate
\[\n u  = \ov \n_i \tau g^{ij}\pard{}{x^j} = (\ov \n \tau)^\top = \ov \n \tau - \psi^{-1} \hat{v} \nu \]
and so
\[|\n u |^2  =   \psi^{-2} (\hat{v}^2-1)\]
as claimed.
\end{proof}

\begin{lemma}\label{dVoldt}
 For any $f\in C^1(M \times [0,T])$ we have
\[\ddt{}\int_M f dV = \int_M \ddt{f} + H^2 f dV\]
\end{lemma}
\begin{proof}
Since at the boundary $\ddt{\mb{F}}\perp \mu$, we do not need to concern ourselves with the manifold flowing ``out'' of $\hat{\Sigma}$. Therefore as is standard we may calculate using Lemma \ref{evolnu}
\[\ddt{} \int_{M_t} f dV =\ddt{} \int_{M^n} f \sqrt{\det g_{ij}}dx = \int_M \ddt{f} + H^2 f dV\] 
\end{proof}

We also require the boundary derivative
\begin{lemma}\label{boundv}
 Let $V$ be a (strictly) timelike eigenvector of the second fundamental form such that $\ov \n_\mu V=0$ and suppose $M_t$ is spacelike. Then at the boundary we have
\[\n_\mu v = -v[A^\Sigma(\nu, \nu) -A^\Sigma(V, V)]\ \ .\]
\end{lemma}
\begin{proof}
 Using Lemma \ref{bdryderivs},  we calculate that 
\begin{flalign*}
 \n_\mu v  = -A(\mu, V^\top) = A^\Sigma(\nu, V^\top)=A^\Sigma(\nu, V-v\nu)=-vA^\Sigma(\nu, \nu)+vA^\Sigma(V, V)
\end{flalign*}
because an eigen vector has the property, $\ov \n_V \mu = \lambda V$ and so $\lambda=-A^\Sigma(V, V)$. Therefore $A(V, \nu) = \lambda \ip{V}{\nu} = vA^\Sigma(V,V)$. 
\end{proof}

\section{Gradient estimates}\label{GE}

Throughout this section we assume Conditions \ref{cassump}, \ref{uassump} and \ref{finitearea} on $\Si$, to obtain the key estimate required for long time existence of the flow, namely the gradient estimate. Firstly we use Condition \ref{cassump} to establish signs on the boundary derivatives of $v$ and $H$, which is a vital step in proving long time existence, compare with similar calculations in \cite{WheelersDirichletNeumann}. We observe that since $\sum_{I=1}^{n-1} (\ip{W_i}{\nu})^2=\size{\nu + \ip{\nu}{ V}V}^2  = v^2-1$, Condition \ref{cassump} implies
\begin{flalign*}
 A^\Sigma(\nu, \nu)-A^\Sigma(V,V)&=\sum_{I} A(W_i, W_i)(\ip{W_i}{\nu})^2 + A^\Sigma(V, V)(v^2-1)\\
&\geq -A^\Sigma(V,V)\sum_I (\ip{W_i}{\nu})^2 +A^\Sigma(V,V)(v^2-1)\\
&=0 \ \ .
\end{flalign*}
As a result, Lemmas \ref{bdryderivs} and \ref{boundv} give that
\begin{equation}
\n_\mu v \leq 0, \ \ \ \ \n_\mu H^2 = -H^2A^\Sigma(\nu,\nu) \leq -H^2A^\Sigma(V,V)\ \ . \label{boundaryderivs} 
\end{equation}

\begin{remark}
 If instead of Condition \ref{cassump} we assume that $\Si$ has merely bounded curvature, the best estimates we may get on the boundary derivatives of $v$ and $H$ are (for some $C(\Sigma)$) $\n_\mu v\leq C v^3$, and $\n_\mu H^2 \leq C H^2 v^2$. This extra factor of $v^2$ adds significant technical problems, with the boundary terms overpowering the evolution equation terms. 
\end{remark}

\begin{remark}
 The gradient estimate we give below depends on a Stampaccia iteration argument (compare \cite{Huiskengraph}\cite{HuiskenConvex}) to get an estimate on $H$. We note that it is also possible to obtain a gradient estimate \emph{without} estimating $H$ using purely maximum principle arguments as in \cite{GerhardtHSurfacesthesequal}. However in an unbounded situation, the methods below give a much better exponent in $u$.
\end{remark}

As is common with Minkowski space problems \cite{BartnikEntire}\cite{EckerMinkowskiDBC}\cite{EckerHuiskenCMCMink} we will estimate $v^{-2}|\n v|^2$ in terms of $|A|^2$ and $H^2$, allowing us to obtain a sign on the evolution of $v$. For this to work, we also need to be able to estimate the extra $H^2$ term by a sufficiently small power of $v$. Unfortunately the boundary derivative of $H^2$ may be positive (when $A^\Sigma(V,V)<0$) and so a direct application of maximum principle does not work. Inspired by the Neumann gradient estimate of Huisken \cite{Huiskengraph} where there were similar problems with the boundary derivative of $v$, we instead use a Stampacchia iteration technique, and to apply this we need Condition \ref{finitearea}. Lemma \ref{dVoldt} then immediately implies that if Condition \ref{finitearea} holds then there exists a finite constant $C(\Sigma)$ which depends on the maximum area of the flowing manifold,\emph{ but is independent of $T$}, such that
\begin{equation}\int_0^T\int_M H^2 dV dt \leq C\ \ . \label{L2Hest}
\end{equation}

We aim to prove:
\begin{prop} Suppose $\Si$ satisfies Conditions \ref{cassump} and \ref{uassump} and \ref{finitearea} and a solution of equation (\ref{MCF}) exists up to some time $T$. Then there exist constants $0<p<1$,  $C_1, C_2>0$ depending only on $n, \Sigma$ and $M_0$ such that
 \label{Hest}
\[ \underset{(x, t)\in M\times[0,T]} \sup |H| \leq C_1 +C_2 \underset{(x, t)\in M\times[0,T]} \sup v^p \ \ .\]
\end{prop}

We introduce the notation \[\mh = \underset{(x, t)\in M\times[0,T]} \sup |H|\qquad\text{ and }\qquad\mv =  \underset{(x, t)\in M\times[0,T]} \sup v\ \ .\] Proposition \ref{Hest} may be proven using the following estimate on the $L^p$ norm of $|H|$ in terms close to $\mv^\frac{p}{2}$ when $p$ is large.
\begin{lemma} \label{LpHest}
 Suppose $\Si$ satisfies Conditions \ref{cassump} and \ref{uassump} and \ref{finitearea} and a solution of equation (\ref{MCF}) exists up to some time $T$. For $k, \gamma>0$ where $k\in \bb{Z}$ and $p=n+2k+\gamma$, there exists a constants $C_1, C_2>0$ depending only on $n, p, \gamma, \Sigma$ and $M_0$ such that
\[ \int_0^T \int_M |H|^p dVdt\leq C_1\mv^{k-1} + C_2 \mv^k \mh^{n+\gamma-2}\ \ .\]
\end{lemma}
\begin{proof}
 Suppose $p>n$ and let $C_n$ be any constant depending on $n,p,\Sigma$ which may change from line to line. By Lemmas \ref{evolnu}, \ref{bdryderivs} and \ref{dVoldt},
\begin{flalign*}
 \ddt{}\int_M |H|^p dV &= \int_{\partial M} -p|H|^pA^\Sigma(V,V)dV^\partial\\
&\qquad\qquad +\int_M -pH^p|A|^2 - p(p-1)H^{p-2}|\n H|^2 +H^{p+2} dV \ \ .
\end{flalign*}
By Lemma \ref{intlemmas} we have that 
\begin{flalign*}\int_{\partial M} -p|H|^pA^\Sigma(V,V)dV^\partial &\leq C_n \int_{\partial M} |H|^pdV^\partial\\
&\leq C_n \int_M|H|^{p-1}|\n H|+|H|^p(v+|A|)  dV 
\end{flalign*}
and so using Young's inequality and $|A|^2\geq\frac{1}{n}H^2$ then
\begin{flalign*}
 \ddt{}\int_M |H|^p dV &
\leq \int_M |H|^{p-2}\left[-(p-n)H^2|A|^2 -p(p-1)|\n H |^2\right. \\
&\qquad\qquad\qquad\qquad\left.+C_n |H||\n H| + C_nH^2(v+|A|)\right]dV\\
&\leq \int_M |H|^{p-2}\left[-\frac{p-n}{2n}H^4 +C_n H^2v\right]dV
\end{flalign*}
and so integrating, 
\[\int_0^T\int_M |H|^{p+2}dVdt\leq C_n \mv\int_0^T\int_M |H|^{p}dVdt + \int_{M_0} |H|^pdV \ \ .\]
Iterating this estimate, we see that for $p$ as described in the statement of the Lemma
\[ \int_0^T\int_M |H|^{p+2}dVdt\leq C_1\mv^k \int_0^T\int_M |H|^{n+\gamma} dV dt + C_2 \mv^{k-1}\]
which completes the proof in light of equation (\ref{L2Hest})
\end{proof}
As is standard for such arguments (see, for example \cite{Huiskengraph}), we will consider the cut-offs of the function $f=H^2$ which we will write as \mbox{$f_k=(H^2-k)_+$}. We define the time dependent set $A(k)=\{x\in M_t: f_k>0\}$, and look to estimate a measure of this set,
\[\|A(k)\|=\int_0^T\int_{A(k)} dVdt\ \ .\]
\begin{lemma}\label{Akest}
 For any $k>0$, there exists a constant $C(k, \Sigma)$ independent of $T$ such that
\[\|A(k)\|\leq C\]
\end{lemma}
\begin{proof}
 \[\|A(k)\|=\int_0^T\int_{A(k)} dVdt\leq \frac{2}{k}\int_0^T\int_{A(\frac{k}{2})} H^2 dVdt\leq \frac{2C}{k} \]
where the constant is from equation (\ref{L2Hest}).
\end{proof}

We will also need the following iteration Lemma:
\begin{lemma}\label{Stampacciait}
 Suppose $\phi:(k_0, \infty) \rightarrow \bb{R}$ is a non--negative non--increasing function such that for all $h>k\geq k_0$ then
\[ \phi(h) \leq \frac{C}{(h-k)^\alpha} (\phi(k))^\beta\]
where $C, \alpha$ and $\beta$ are positive constants. Then if $\beta>1$ then $\phi(k_0+d)=0$ for
\[ d^\alpha = C [\phi(k_0)]^{\beta-1} 2^{\alpha\frac{\beta}{\beta-1}}\ \ .\]
\end{lemma} 
\begin{proof}
See \cite[Lemma 4.1 i)]{Stampacchia}.
\end{proof}
We now prove the Proposition:
\begin{proof}[Proof of Proposition \ref{Hest}]
We look at the evolution of $f_k^p$ for some large $p>\frac{n}{2}$. From Lemma \ref{evolnu} and (\ref{boundaryderivs}),
\[\ho f_k^p = pf_k^{p-1}\left[-2H^2|A|^2 -2|\n H|^2\right]-p(p-1)f_k^{p-2}4H^2|\n H|^2\]
\[\n_\mu f_k = -pf_k^{p-1}H^2A^\Sigma(V,V)\leq C_n f_k^{p-1}H^2\ \ .\]
Therefore using Lemma \ref{intlemmas} we have:
\begin{flalign*}
 \int_{\partial M} \n_\mu f_k^p dV^\partial &\leq C_n \int_{\partial M}f_k^{p-1}H^2 dV^\partial \\
&\leq C_n \int_M f_k^{p-2}|H|^3|\n H| + f_k^{p-1} |H| |\n H| +  f_k^{p-1}H^2 (v + |A|)dV\ \ .
\end{flalign*}
Estimating similarly to in Lemma \ref{LpHest}, (and using that $2p>n$) 
\begin{flalign*}
 \ddt{}\int_M f_k^pdV &\leq \int_M  pf_k^{p-1}\left[-2H^2|A|^2 -2|\n H|^2 +C_n|H| |\n H|+C_nH^2 (v + |A|)\right]\\
&\qquad+p(p-1)f_k^{p-2}\left(-4H^2|\n H|^2+C_n|H|^3|\n H|\right) + H^2 f_k^p dV\\
&\leq \int_M  pf_k^{p-1}\left[C_nvH^2  -|\n H|^2 \right]\\
&\qquad+p(p-1)f_k^{p-2}\left(-2H^2|\n H|^2+C_n|H|^4\right) dV \ \ .
\end{flalign*}
We have that $|\n f_k^p|\leq f_k^{p-1}\left[\frac{C_n}{\epsilon} H^2 +\epsilon |\n H|^2\right]$, and so 
\begin{flalign*}
 \ddt{}\int_M f_k^pdV & \leq \int_M C_nf_k^{p-2}H^4v dV- C_n\int_M |\n f_k^p| +f_k^p dV\\
&\leq C_n\mv \int_{A(k)} H^{2p} dV - \frac{C_n}{\mv}\left(\int_M f_k^\frac{np}{n-1}\right)^\frac{n-1}{n}\ \ .
\end{flalign*}
We now set $k>k_0=\underset{x\in M_0} \sup H^2$ and integrate to get
\[\underset{t\in[0,T]}\sup \int_M f_k^p dV+ \frac{C_n}{\mv} \int_0^T \left(\int_M f_k^\frac{np}{n-1}\right)^\frac{n-1}{n} dt \leq C_n\mv \int_0^T\int_{A(k)} H^{2p} dVdt\ \ .\]
By standard methods,
\[\underset{t\in[0,T]}\sup \int_M f_k^p dV+ \frac{C_n}{\mv} \int_0^T \left(\int_M f_k^\frac{np}{n-1}\right)^\frac{n-1}{n} dt \geq \frac{C_n}{\mv^\frac{n}{n+1}}\frac{\int_0^T\int_{A(k)}f_k^p dV dt}{\|A(k)\|^\frac{1}{n+1}}\]
and so by H\"older's inequality, 
\begin{flalign*}|h-k|^p\|A(h)\|&\leq \int_0^T\int_{A(k)}f_k^p dV dt \\
&\leq  C_n \mv^{1+\frac{n}{n+1}}\left(\int_0^T\int_M H^\frac{2p}{\epsilon}\right)^\epsilon \|A(k)\|^{1-\epsilon + \frac{1}{n+1}}
\end{flalign*}
We now set $\epsilon = \frac{1}{2(n+1)}$, let $j \in \bb{Z}$ be so large that $p>2$ where $\frac{2p}{\epsilon} = n+1+2j$. By Lemma \ref{LpHest},
\[|h-k|^p\|A(h)\| \leq C_n \mv^{1+ \frac{n}{n+1}}\left(\mv^{j-1} + \mv^j \mh^{n-1}\right)^\epsilon\|A(k)\|^{1+\frac{1}{2(n+1)}}\]
Therefore from Lemma \ref{Stampacciait}, Lemma \ref{Akest} we see that $\|A(k_0+1 +d)\|=0$ for particular $d$ depending on $\mv$ and $\mh$. Explicitly, we may estimate:
\begin{flalign*}
 \mh^2 &\leq k_0+1 + C_n \mv^\frac{1+\frac{n}{n+1}}{p}\left(\mv^{j-1} + \mv^j \mh^{n-1}\right)^\frac{\epsilon}{p}\\
&\leq  k_0+1 + C_n \mv^{\frac{2n+1}{n+1+2j} +\frac{2j}{n+1+2j}}\left( 1+ \mh^\frac{2n-2}{n+1+2j}\right) \ \ .
\end{flalign*}
The Proposition is now proved by making $j$ very large.
\end{proof}

We may now use standard methods to obtain a gradient estimate which is exponential in a height function $u$.

\begin{theorem}
Suppose $\Si$ satisfies Conditions \ref{cassump}, \ref{uassump} and \ref{finitearea}. Then there exist constants $C_1, C_2>0$ depending on $n, \Sigma$ and $M_0$ but independent of time such that for all the time the flow exists \label{Gradest}
\[v\leq  C_1e^{C_2 \text{osc}\, u}\underset{x\in M_0}\sup v\ \ .\]
\end{theorem}
\begin{proof}
We consider the function $f = ve^{\lambda u}$. Using Lemma \ref{evolu} , 
\[\ho e^{\lambda u }  =  e^{\lambda u} \left( -2\lambda g^{ij}\ov \n^2_{ij} u -2\l^2 |\n u |^2\right)
\]
We estimate
\begin{flalign*}
\frac{|\n v|^2}{v^2}
&\leq (1+\epsilon_1)A(\frac{V^\top}{v}, i)g^{ij}A(\frac{V^\top}{v}, j) + C_n(1+\frac{1}{\epsilon_1})v^2 
\end{flalign*}
therefore since $|V^\top|^2 \leq \sqrt{v^2-1}$, we may estimate as in \cite[Theorem 3.1]{BartnikEntire}
\begin{flalign*}
 |A|^2
&\geq(1+\frac{1}{n})\left(\frac{1}{1+\epsilon_1}\frac{|\n v|^2}{v^2} - \frac{C}{\epsilon_1}v^2 \right) - H^2
\end{flalign*}
We use these inequalities and Lemma \ref{evolv} to obtain that for $\e_1$ and $\e_2$ small, 
\begin{flalign*}
\ho f &= ve^{\l u}\left( -|A|^2 +2v^{-1}g^{ij} A(\n_iV, j) +v^{-1}g^{ij}\ip{\ov{\n}_{ij} V}{\nu}\right.\\
&\qquad\qquad\qquad\qquad\qquad\left.-2\lambda \ip{\frac{\n v}{v}}{\n u} -2\lambda g^{ij}\ov \n^2_{ij} u -\l^2|\n u |^2\right)\\
&\leq  ve^{\l u}\left( -(1-\epsilon_2)|A|^2 +C_n\left(1+\l+ \frac{1}{\epsilon_2}\right)v^2 -2\lambda \ip{\frac{\n v}{v}}{\n u} -\l^2|\n u |^2\right)\\
&\leq  ve^{\l u}\left(-(1-\epsilon_2)\left(\frac{1 +\frac{1}{n}}{1+\epsilon_1}\right)\frac{|\n v|^2}{v^2}+H^2 + C_n\left(1+\l+ \frac{1}{\epsilon_1}+\frac{1}{\epsilon_2}\right)v^2\right.\\
&\left.\qquad\qquad -2\lambda \ip{\frac{\n v}{v}}{\n u}  -\l^2|\n u |^2\right)\\
&\leq  ve^{\l u}\left( H^2 + C_n\left(1+\l+ \frac{1}{\epsilon_1}+\frac{1}{\epsilon_2}\right)v^2 -\l^2(1 - \frac{1+\epsilon_1}{(1-\epsilon_2)(1+ \frac{1}{n})})|\n u |^2\right)\ \ .
\end{flalign*}
Choosing, for example, $\epsilon_1 = \frac{1}{4n}$ and $\epsilon_2$ so that $(1-\epsilon_2)(1+\frac{1}{n}) = 1+\frac{1}{2n}$, then using Proposition \ref{Hest} and Lemma \ref{evolu}, 
\begin{flalign*}
 \ho f&\leq  ve^{\l u}\left( \mv^{2p} + C_nv^2   + C_n \l v^2 -\frac{\l^2}{4n+2}|\n u |^2\right)\\
&\leq  ve^{\l u}\left(\mv^{2p} +C_nv^2   + C_n \l v^2 -\frac{\l^2\psi^{-2}}{4n+2}(\hat{v}^2 - 1)\right)\ \ .
\end{flalign*}
Therefore due to the uniform lower bound on $\psi$, and the equivalence of $v$ and $\hat{v}$, when $v>2\tilde{C}_V$ (where $\tilde{C}_V$ is the constant from Remark \ref{remarkablev}) we may choose $\lambda$ sufficiently large, to obtain on the interior of $M_t$
\[\ho f \leq ve^{\lambda u}\left(\mv^{2p} - v^2\right) \]
while meanwhile at the boundary, due to Condition \ref{cassump}, and Lemma \ref{boundv}
\[\n_\mu f \leq 0\ \ .\]

We now apply a maximum principle argument to remove the possibility of large increasing maxima of $f$ when $v>2\tilde{C}_V$.

At an increasing maximum $(p, s)$ of $f$, where $f(p,s) = \underset{(x,t)\in M^n \times[0,s]}\sup f(x,t)$,  then for $\mv(s) = \underset{(x,t)\in M^n \times[0,s]}\sup v(x,t)$ we have
\[0\leq \mv^{2p} - v^2 .\]
For $(x, t)\in M^n\times [0,T]$ write $m\leq e^{\lambda u}\leq M$, then for any $(q,r)\in M^n \times [0,s]$ such that $v(q,r)\geq \frac{\mv(s)}{2}$,
\[\frac{m}{2} \mv(s) \leq e^{\lambda u(q,r)}v(q,r) \leq (e^{\lambda u }v)(p,s) \leq M \mv^p(s)\ \ ,\]
therefore $\mv(s) \leq \left(2 \frac{M}{m}\right)^\frac{1}{1-p}$, and $f\leq \left(2 \frac{M}{m}\right)^\frac{1}{1-p} M$. Therefore an increasing interior maximum is bounded by exponents of $u$.

At the boundary if $v^2 \geq \mv^{2p}$ then we may apply the elliptic Hopf lemma (see for example \cite[Lemma 3.4, p34]{Gilbarg}) to disallow an increasing boundary maximum. Otherwise we obtain exactly the situation above. 

Therefore we have $f\leq \max\left\{ \underset{M_0} \sup \, ve^{\lambda u}, \  \left(2 \frac{M}{m}\right)^\frac{1}{1-p} M, \ \tilde{C}_VM \right\}$. We observe that adding a constant function to $u$ changes nothing above, and so without loss of generality we may assume that $m=1$. The estimate on $f$ implies the theorem.
\end{proof}

\begin{cor}\label{boundedflow}
Suppose $\Si$ satisfies Conditions \ref{cassump} and \ref{uassump} and there exists comparison solutions such that $\underline{C_u}\leq u\leq \ov{C_u}$. Then a smooth solution to equation (\ref{MCF}) exists for $T=\infty$ for which for all $\epsilon>0$, we have the uniform estimate \[\underset{M^n\times [\e, \infty)}\sup |\n^k A|\leq C_k(\e)\ \ .\] 
\end{cor}
\begin{proof}
The gradient estimate, Theorem \ref{Gradest} shows that equation (\ref{graphMCF}) is a uniformly parabolic quasilinear equation with with a linear boundary condition. Therefore by applying standard quasilinear parabolic theory, see for example \cite{Lieberman}, we have existence of a smooth solution for $T=\infty$. The uniform parabolic norms estimate follows from the fact that we have a uniform estimate on the gradient and height. 
\end{proof}

\begin{cor}\label{unboundedflow}
Suppose $\Si$ satisfies Conditions \ref{cassump} and \ref{finitearea}. Then any for any smooth compatible initial data, a solution to (\ref{MCF}) either exists for $T=\infty$, or $u$ becomes unbounded in finite time. 
\end{cor}
\begin{proof}
 As Condition \ref{finitearea} holds for all time $u$ is bounded, the Corollary follows from Theorem \ref{Gradest}.
\end{proof}

\section{Convergence and stability}\label{CT}
We now look into questions of convergence when $\mb{F}$ stays in a bounded region.
\begin{lemma}
 If $\Sigma$ is as in Corollary \ref{boundedflow}, then there exists a sequence of times $t_k\rightarrow \infty$ such that $M_{t_k}$ tends towards $M_\infty$ in the $C^\infty$ topology where $M_\infty$ is a maximal surface satisfying the boundary condition.
\end{lemma}
\begin{proof}
This is as in \cite[Proof of Theorem 4.2]{EckerHuiskenCMCMink}. 
\end{proof}

Convergence of the whole flow is not so straightforward and is related to stability of the maximal surfaces towards which the flow converges. This stability depends on the geometry of $\Sigma$ close to the maximal surface. To illustrate this we consider rotationally symmetric $\Sigma$.
\begin{lemma}\label{rot}
 Let $\Sigma\subset \bb{R}^3_1$ be a smooth rotationally symmetric boundary manifold, parametrised by $\mb{E}(z, \theta) = f(z)\mb{r} + z\mb{e_3}$ where $\mb{r}=\cos \theta\,\mb{e}_1+\sin \theta\,\mb{e}_2$ such that $f>0$ and $|f'(z)|<1$. Then $\Sigma$ satisfies Condition \ref{cassump} if and only if
\begin{equation}f''f\leq1-(f')^2\ \ .  \label{rotassump}\end{equation}
\end{lemma}
\begin{proof}
 We may calculate in these coordinates
\[\mu = \frac{1}{\sqrt{1-(f')^2}}(\mb{r}+ f'\mb{e}_3), \ \ h^\Sigma_{zz}=-\frac{f''}{\sqrt{1-(f')^2}}, \ \ h^\Sigma_{z\theta}=h^\Sigma_{\theta z}=0, \ \ h^\Sigma_{\theta\theta}=\frac{f}{\sqrt{1-(f')^2}} \ \ .\]
Therefore the principle directions are $V=\frac{f' \mb{r} +e_3}{\sqrt{1-(f')^2}}$ and $W= \mb{r}_\theta$ which gives
\[A^\Sigma(V,V)=\frac{-f''}{(1-(f')^2)^\frac{3}{2}} , \ \ A^\Sigma(W, W)=\frac{1}{f\sqrt{1-(f')^2}}\]
and Condition \ref{cassump} becomes equation (\ref{rotassump})
\end{proof}

We may obtain Conditions \ref{finitearea} and \ref{uassump} on such a rotational $\Sigma$ by, for example, assuming $f$, $f'$ and $f''$ are uniformly bounded and smooth.
\begin{example}
In the extreme case of the above, where $f''f=1-(f')^2$ everywhere, then we may integrate to get for arbitrary $A,B$
\[f(x)= \sqrt{A^2 + (z+B)^2}\]
or we obtain the pseudo-sphere in $\bb{R}^3_1$, i.e. the set of points $x\in\bb{R}^3_1$ such that $|x - Be_3|^2 = A^2$. We remark that in this case, comparison solutions constructed in Lemma \ref{moregravy} move off towards infinity, and so we do not necessarily expect convergence to a maximal surface. However, we are still able to apply Corollary \ref{unboundedflow} to obtain long time existence of the flow.
\end{example}
\begin{lemma}\label{gravy}
 If $\Sigma$ is as in Lemma \ref{rot} and satisfies (\ref{rotassump}), then $\hat \Si$ admits a foliation of constant mean curvature surfaces, where each leaf is a plane or a hyperbolic plane which satisfies the perpendicular boundary condition.
\end{lemma}
\begin{proof}
We aim to do this by constructing constant mean curvature foliation of planes and hyperbolic planes of $\hat{\Sigma}$. A general hyperbolic plane may be written $P(l,\theta)= R (\sinh l \mb{r} + \cosh l e_3) + Je_3$. We suppose that such a $P$ perpendicularly intersects a rotational surface $\Sigma$ at points $f(z)\mb{r}(\theta)+ze_3$ for some fixed $z$, that is $\cosh l \mb{r} + \sinh l \mb{e}_3=\mu(z)$. This gives that if $f'\neq 0$, 
\[R = \frac{f}{f'}\sqrt{1-(f')^2}, \ \ J=z-\frac{f}{f'}\]
and so we define
\[P(l,\theta, z)= \frac{f}{f'}\sqrt{1-(f')^2}(\sinh l \mb{r} + \cosh l e_3) + (z-\frac{f}{f'})e_3\ \ .\]
This represents a foliation if the leaves of the foliation do not cross, and since these are rotationally symmetric, this is equivalent to not crossing at $l=0$. Therefore we have a foliation if $\pard{g}{z}>0$ where 
\[g(z) = -\ip{P(0, \theta, z)}{e_3} = z - \frac{f}{f'}(1-\sqrt{1-(f')^2})\ \ .\]
We calculate
\begin{flalign*}
 g' &= \sqrt{1-(f')^2}\left[1-\frac{f''f}{ (1+\sqrt{1-(f')^2})(1-(f')^2)}  \right]\\
\end{flalign*}
From equation (\ref{rotassump}), we have
\[1-\frac{f''f}{ (1+\sqrt{1-(f')^2})(1-(f')^2)}\geq 1-\frac{1}{ (1+\sqrt{1-(f')^2})}>0\ \ .\] 
Therefore, we may always obtain a foliation of CMC surfaces if we have Condition \ref{cassump} and $f'>0$. When $f'\rightarrow 0$, $g'\geq \frac{1}{2}>0$, and so the leaves never cross. In this case the above parametrisation becomes degenerate, but the hyperbolic planes converge to a maximal plane.
\end{proof}

If $f'(c)=0$ at a point then there exists a planar maximal surface at height $c$ given by $\tilde{M}=\{(x, y, z)\in\hat{\Si}\subset \bb{R}^{3}_1| z=c\}$, satisfying equation (\ref{MCF}). 
\begin{definition}
 A solution to mean curvature flow $\mb{F}$ is said to be \emph{stable under the flow} if for any sufficiently small perturbation $\tilde{\mb{F}}_0$ (which still satisfies the compatibility condition) of the initial embedding $\mb{F}_0$, the perturbed flow will converge uniformly to $\mb{F}$ as $t\rightarrow \infty$. 
\end{definition}

We now look at stability of such surfaces:

\begin{lemma} \label{moregravy}
  If $\Sigma$ is as in Lemma \ref{rot} and satisfies (\ref{rotassump}), and suppose that $f'(c)=0$. Then there exist comparison solutions pushing solutions of (\ref{MCF}) away from planar maximal hypersurfaces at height $c$ satisfying $f''(c)>0$ and towards the planar maximal surfaces with $f''(c)<0$ at the boundary.   
\end{lemma}
\begin{proof}
We obtain comparison solutions from the foliations in Lemma \ref{gravy}. Take $\Omega$ to be the unit disk $D$, $(r, \theta)$ polar coordinates on $D$ and $P$ as in Lemma \ref{gravy}. Write $K(z)$ for the length of the radial geodesic from $P(0,0, z)$ to $\Sigma$ along the hyperbolic plane determined by $P(\cdot, \cdot,z)$. We define $\mb{G}:D\times[0,T)\ra \bb{R}^{n+1}_1$ by $\mb{G} = P(rK(z(t)), \theta, z(t))$ where $z(t):[0,T)\ra \bb{R}$ is to be determined. Locally, we choose the normal to the foliation so that on any hyperbolic plane $H>0$, and choose the parametrisation so that $R>0$, where $R$ is as in Lemma \ref{gravy}. We see that if $f'\neq0$, equation (\ref{CompSol}) is then equivalent to
\begin{equation}\dot{z}\left[\pard{R}{z} + \pard{J}{z}\cosh(Kr)\right]\leq \frac{2}{R(z)} \ \ .
\label{diffComp} 
\end{equation}
A solution to (\ref{diffComp}) will be a comparison solution to a solution of (\ref{MCF}) for which $M_0$ is on the side of $\mb{G}$ into which the normal points.

We have
\begin{flalign*}
 \pard{R}{z} + \pard{J}{z}\cosh(Kr)
 & = \sqrt{1-(f')^2} +\frac{\left(\sinh^2(Kr)-(f')^2\cosh^2(Kr)\right)ff''}{(f')^2\sqrt{1-(f')^2}(\sqrt{1-(f')^2}\cosh(Kr) + 1)}.
 \end{flalign*}

 We define $w(r):= \sinh^2(Kr)-(f')^2\cosh^2(Kr)$ and we may calculate as in Lemma \ref{gravy} that $\sinh(K) = \frac{f}{R} = \frac{f'}{\sqrt{1-(f')^2}}$. We therefore see that $g(0) = -(f')^2$, and $g(1) = 0$. Further simple calculations give that $g'\geq0$. We therefore see that $w(r)\in[-(f')^2, 0]$ for $r\in[0,1]$.
 
 If $f''\geq0$, we use (\ref{rotassump}) to obtain
 \begin{flalign*}
  \pard{R}{z} + \pard{J}{z}\cosh(Kr) &\geq \sqrt{1-(f')^2} -\frac{ff''}{\sqrt{1-(f')^2}(\sqrt{1-(f')^2}\cosh(Kr) + 1)}\\
 &\geq \sqrt{1-(f')^2}\left(1- \frac{1}{\sqrt{1-(f')^2}\cosh(Kr) + 1}\right)\\
 &>0
 \end{flalign*}

 If $f''<0$ then $\pard{R}{z} + \pard{J}{z}\cosh(Kr) \geq \sqrt{1-(f')^2}>0$. 

Equation (\ref{diffComp}) therefore yields an ordinary differential inequality in $z$ which may be solved to obtain comparison solutions moving in the $e_3$ direction. Observing that for leaves close to a maximal surface, the sign on  $H^G = 2R^{-1}$ of the foliation is determined by $f''$, the claimed stability and instability follow.
\end{proof}

\begin{figure}[h!]\label{CMCfoliation}
\includegraphics[scale=0.28]{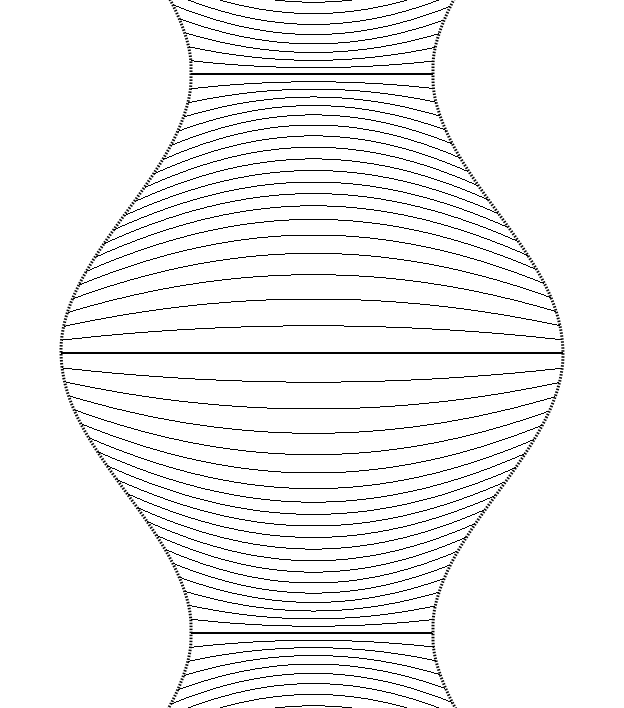}\includegraphics[scale=0.28]{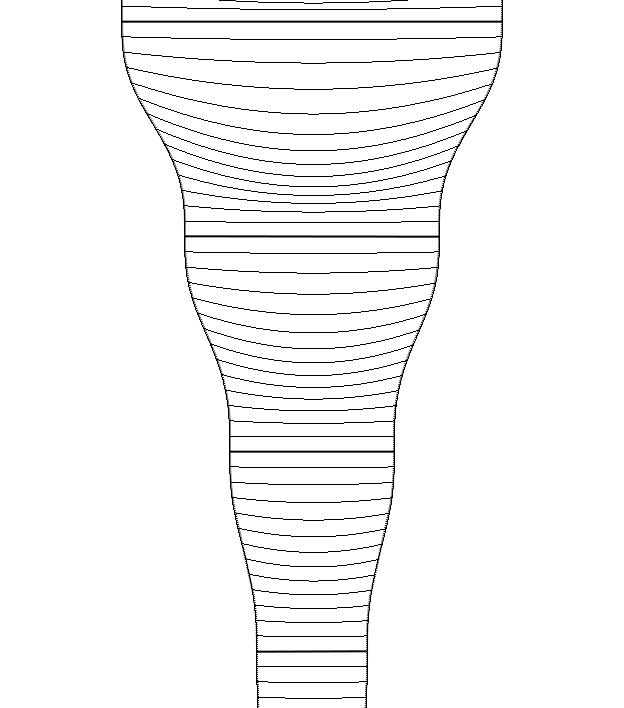}
 \caption{Two examples of foliations by CMC surfaces, demonstrating stability and instability of maximal planes.}
\end{figure}

In Figure \ref{CMCfoliation} we see three examples of possible stability behaviour of planar maximal surfaces. The left picture shows one completely stable plane at the widest point of the sine wave, and two unstable planes at the thinnest points.  We remark that since the plane is a maximal surface, and therefore a comparison solution Proposition \ref{Comparison} implies that MCF starting at a one-sided perturbation of one of the the unstable maximal surfaces will move away towards the stable maximal surfaces. The right hand picture shows examples with one sided stability -- perturbations on the lower side will flow back towards the maximal surface while flowing a one-sided upwards perturbation will move away towards a higher maximal surface.

It is also easy to see that despite the existence of a comparison solution moving away from the the unstable maximal surfaces in the left picture, there exist solutions to MCF which must intersect this maximal surface for all time. For example if we were to perturb by a two sided perturbation which is rotationally symmetric around the $y$-axis, the solution must always intersect the unstable plane due to preservation of symmetry by the flow. If there are no other maximal surfaces intersecting the plane, a subsequence of the flow must converge to the unstable maximal surface.

\begin{remark}
 Variational stability of a maximal surface \emph{does not} imply stability under the flow. We may observe this by taking a convex cylindrical $\Sigma$ and considering graphical MCF where, as in \cite{LambertConvex}, the flow then converges to planes given graphically by $u=\text{const}$. The condition for variational stability (where we assume perturbations also satisfy the boundary condition) becomes
\[2\int_{\partial M} \phi^2 A^\Sigma(\nu,\nu) dV^\partial + 2\int_M |\n \phi|^2+ \phi^2 |A|^2 dV\geq 0\]
for any function $\phi$ such that $\n_\mu \phi^2 = -2\phi^2 A^\Sigma(\nu,\nu)$, which is trivially true for constant graphs inside a convex cylinder, $\Si$. But from Proposition \ref{Comparison}  a one sided perturbation of such a maximal surface will converge to a \emph{different} maximal surface, and so we do not have stability under MCF. 

Stability of the flow \emph{does} imply variational stability. For sufficiently small variations of a maximal hypersurface which is stable under the flow, MCF will move the surface back to the maximal hypersurface. We therefore see that any small perturbation cannot have $H \equiv 0$ everywhere, and so by Lemma \ref{dVoldt} the volume of the flowing surface strictly increases under the flow, and the maximal surface is variationally stable. 
\end{remark}

We give a condition for stability under the flow.
\begin{lemma}\label{FullStab}
Suppose $\Si$ satisfies Condition \ref{uassump} and $\tilde{M}$ is a smooth compact uniformly spacelike maximal surface with boundary $\partial \tilde{M}$ and normal $\nu$, where $\partial \tilde{M}$ satisfies the boundary condition that $\partial \tilde{M}\subset \Si$ and $\ip{\nu}{\mu}=0$. Suppose there exists a $\tau>0$ and a $\phi:\tilde{M}\ra \bb{R}$, $\phi\geq 2\tau$ such that \label{Stab1}
 \begin{equation}\label{Stabcondition}
 \begin{cases}
 \Delta \phi - \phi |A|^2  <0\\
 \n_\mu \phi  \geq - \phi A^\Sigma(\nu,\nu)
 \end{cases}
 \end{equation}
 then $\tilde{M}$ is stable. 
\end{lemma}
\begin{proof}
 We construct a comparison solution from above, comparison solutions from below follow identically. 
Let $\tilde{F}:\Omega\ra \bb{R}^{n+1}_1$ parametrise $\tilde{M}$, and let $\ov \nu$ be a local extension of $\nu$ the normal of $\tilde{M}$ to an open neighbourhood of $\tilde{M}$ in $\bb{R}^{n+1}_1$ such that $|\ov \nu|^2=-1$, and for $p\in \Si$, $\ov \nu (p) \in R_p\Si$. Then define $J:\Omega\times(-\epsilon_1, \epsilon_1) \ra \bb{R}^{n+1}_1$ by the differential equation
\[\pard{J}{\l}(x,\l) = \ov \nu (\tilde{F}(x)+J(x,\l))\ \ , \qquad J(x,0)=0\ \ .\]
We see that $K:\Omega\times(-\epsilon_1, \epsilon_1) \ra \bb{R}^{n+1}_1$ defined by $K(x,\l)=\tilde{F}(x)+\phi(x)J(x,\l)$ is locally a diffeomorphism and $K(\partial \Omega)\subset \Si$.

We consider how geometric quantities vary on the hypersurfaces given by $K(\cdot, \l)$. Identically to the proof of Proposition \ref{evolnu} we calculate
 \[\pard{}{\l}|_{\l =0} g_{ij} = 2\phi h_{ij},  \ \ \pard{\nu}{\l}|_{\l=0} = \n \phi\ \ .\]
We also have
 \begin{flalign*}
  \pard{}{\l}|_{\l=0} h_{ij}& = -\pard{}{\l}|_{\l=0}\ip{\nu}{\frac{\partial^2(\tilde{F}+J)}{\partial x^i \partial x^j}}  = \n_{ij}^2 \phi +\phi h_i^kh_{kj}\ \ ,
 \end{flalign*}
and so
\[\pard{H}{\l}|_{\l=0} = -h^{ab}\pard{g_{ab}}{\l}|_{\l=0} + g^{ij}\pard{h_{ij}}{\l}|_{\l=0} = \Delta \phi -\phi |A|^2\ \ .\]
 
 At the boundary we see that 
 \[\pard{}{\l}\ip{\nu}{\mu}|_{\l =0} = \n_\mu \phi + A^\Sigma(\nu,\nu)\phi \ .\]
 
Due to the compactness of $M$ and (\ref{Stabcondition}) we see that  there is a $\delta>0$ such that $\pard{H}{\l}|_{\l=0} < -2\delta$. We also observe $\pard{}{\l}\ip{\nu}{\mu}|_{\l =0}\geq0$ and $\ip{\pard{J}{\l}}{\nu}|_{\l=0} = -\phi\leq -2\tau$. We use continuity of the above quantities to see that for $\l\in[0, 2\e_2)$, 
\[\underset{x\in\Omega}\sup H(x,\l)<-\l\delta, \qquad \underset{x\in\partial M}\inf\ip{\nu}{\mu}(x, \l)\geq 0, \qquad \underset{x\in\Omega}\sup\ip{\pard{J}{\l}}{\nu}(x, \l)\leq -\tau\ \ . \]

We now write $\mb{G}(x,t) = K(x, s(t))$, and bearing in mind that $\nu^G=-\nu$ we see that $\mb{G}$ satisfies (\ref{CompSol}) for $s(t)<2\e_2$ if:
\[\ip{\frac{d \mathbf{G}}{dt}}{\nu^G}=-\dot{s}\phi \ip{\pard{J}{\l}}{\nu} \geq H(x,s(t))=-H^G\]
This is implied by
\[\dot{s}\geq -\frac{\delta}{\tau^2}s(t)\]
Therefore there exists a very small $\theta>0$ such that $\mb{G}(x,t)=K(x, \e_2 e^{-\theta t})$ is an upper comparison solution which converges back to $M$ as $t\ra \infty$. 
 \end{proof}

\begin{cor} \label{Stabcondition2}
 If $\tilde{M}$ is as in the previous Lemma and also at every point $p\in \partial \tilde{M}$,  $A^\Sigma(\nu,\nu)|_p>0$, then $\tilde{M}$ is stable.
\end{cor}
\begin{proof}
 Pick a point $a\in \bb{R}^{n+1}_1$ and consider the function $f=R-|x-a|^2$ which we will show satisfies (\ref{Stabcondition}) for $R$ large enough. 
 
 We may easily see that $\ov{\n}^2_{ij} \phi = -2\ov{g}_{ij}$ and so, since $\tilde{M}$ is maximal $\Delta \phi= -2n$. Therefore, the first equation in (\ref{Stabcondition}) follows if $\phi$ is positive.
 
 At the boundary we have
 \[\n_\mu \phi = -2\ip{x-a}{\mu}\ \ .\]
 By compactness of $\partial \tilde{M}$, $A^\Sigma(\nu,\nu)>\delta$ for some $\delta>0$, and similarly (by uniform spacelikeness of $\tilde{M}$ and compactness of $\partial \tilde{M}$) $\ip{\mu}{x-a}$ is bounded above. Therefore there exists a $R_0>0$ such that for all $R\geq R_0$, \[2\ip{x-a}{\mu}\leq A^\Sigma(\nu,\nu)(R-|x-a|^2)\ \ .\]
 
Setting $R=\max\{R_0, \underset{M} \sup |x-a|^2 +\tau\}$ then (\ref{Stabcondition}) holds and so by Lemma \ref{FullStab} we are done. 
\end{proof}

%   Ben Lambert\\
%              University of Konstanz\\
%	      Zukunftskolleg\\
%	      Box 216\\
%	      78457 Konstanz\\
%	      Germany \\
%              \email{benjamin.lambert@uni-konstanz.de}}
\end{document}